\documentclass[12pt]{amsart}
\usepackage{amsfonts,amssymb,amsthm,amscd}

\newtheorem{theorem}{Theorem}

\newtheorem{lemma}{Lemma}

\newtheorem{proposition}{Proposition}
\newtheorem{corollary}{Corollary}

\theoremstyle{definition}
\newtheorem{definition}{Definition}
\theoremstyle{remark}

\newtheorem{remark}{Remark}

\newcommand{\aaa}{\mathfrak{a}}
\newcommand{\res}{\operatorname{res}}
\newcommand{\mix}{\mathcal{M}}

\newcommand{\abs}[1]{\left|{#1}\right|}

\newcommand{\sprod}[2]{\left\langle#1,#2\right\rangle}

\newcommand{\ff}{\mathbb{F}}
\newcommand{\bbc}{\mathbb{C}}

\newcommand{\oo}{\mathcal{O}}

\newcommand{\ccc}{\mathcal{C}}

\newcommand{\bs}{\backslash}
\newcommand{\Hom}{\operatorname{Hom}}

\newcommand{\diag}{\operatorname{diag}}

\newcommand{\zz}{\mathbb{Z}}

\newcommand{\cc}{\mathbb{C}}

\newcommand{\rr}{\mathbb{R}}

\newcommand{\A}{\mathbb{A}}

\author{Omer Offen and Eitan Sayag}

\address{Humboldt-Universit\"at zu Berlin, Mathematich-Naturwissenschaftliche Fakult\"at II,
Institut f\"ur Mathematik, Sitz: Rudower Chausee 25, D-10099 Berlin,
Germany.} \email{offen@mathematik.hu-berlin.de}

\address{Einstein Institute of Mathematics,
Edmond J. Safra Campus, Givat Ram, The Hebrew University of
Jerusalem, Jerusalem, 91904, Israel} \email{sayag@math.huji.ac.il}

\title[Mixed Periods]{Global Mixed Periods and local Klyachko models for the general linear
group}

\begin{document}
\begin{abstract}
We show that every irreducible representation in the discrete
automorphic spectrum of $GL_n(\A)$ admits a non vanishing mixed
(Whittaker-symplectic) period integral. The analog local problem is
a study of models first considered by Klyachko over a finite field.
Locally, we show that for a $p$-adic field $F$ every irreducible,
unitary representation of $GL_n(F)$ has a Klyachko model.
\end{abstract}

\maketitle
\section{Introduction}\label{introduction}

Fundamental to the theory of automorphic forms on $GL_n$ is the fact
that a cuspidal automorphic representation admits a global Whittaker
functional. Other period integrals were considered for certain
representations in the residual spectrum. The study of global
symplectic period integrals for $GL_n$ was initiated by Jacquet and
Rallis in \cite{MR93b:22035}. They were further studied in
\cite{OII}. In \cite{Odist}, the first named author characterized
all irreducible representations in the discrete automorphic spectrum
that admit a symplectic period. The main global result of the
present work provides a non zero period integral for any irreducible
representation in the discrete spectrum. Namely, following
\cite{MR1216185}, we consider a certain finite list of
Whittaker-symplectic period integrals and show that every discrete
spectrum representation of $GL_n$ admits one of them.

The mixed period integral is factorizable (see Corollary
\ref{factorizable}) and our global results have local analogues. The
local results of this work continue the study of symplectic models
considered in \cite{MR91k:22036, mix}. In \cite{kly}, Klyachko
introduced certain mixed (Whittaker-symplectic) models in the
context of $GL_n$ over a finite field. Our main local result extends
the work of Heumos and Rallis. We show that every irreducible,
unitary representation of $GL_n$ over a $p$-adic field has a
Klyachko model. This was previously obtained in \cite{MR91k:22036}
for $n \le 4$. See also \cite{Nien}.

To describe our results more precisely, we set the necessary
notation. Let $F$ be either a number field or a $p$-adic field. In
the global case, denote by $\A=\A_F$ the ring of ad\`eles of $F$.
Let $G_r=GL_r$ be regarded as an algebraic group defined over $F$
and let $U_r$ denote the group of upper triangular unipotent
matrices in $G_r$.

Fix $n$ and let $G=G_n$. For any decomposition $n=r+2k$ we
consider a subgroup of $G_{n}$ defined by
\[
    H_{r,2k}=\{
    \left(
        \begin{array}{cc}
          u & X \\
          0 & h \\
        \end{array}
    \right)
 \in G:u \in U_r,\,X \in M_{r \times 2k} \text{ and } h \in
 Sp(2k)\}.
\]
Here
\[
    Sp(2k)=\{g \in G_{2k}:{}^t g
    \left(
        \begin{array}{cc}
            & w_k \\
            -w_k &
        \end{array}
    \right)g=
    \left(
        \begin{array}{cc}
            & w_k \\
            -w_k &
        \end{array}
    \right)\}
\]
and $w_k \in G_k$ is the permutation matrix whose $(i,j)$th entry is
$\delta_{k+1-i,j}$. Let $\psi$ be a non trivial character of $F$ in
the local case (resp. of $F \bs \A$ in the global case). We
associate to $\psi$ the character $\psi_r$ of $U_r(F)$ (resp.  of
$U_r(F) \bs U_r(\A)$) defined by
\[
    \psi_r(u)=\psi(u_{1,2}+\cdots+u_{r-1,r}).
\]
By abuse of notation we will also denote by $\psi_r$ the character
of $H_{r,2k}(F)$ (resp. of $H_{r,2k}(F) \bs H_{r,2k}(\A)$) defined
by
\[
    \psi_r     \left(
        \begin{array}{cc}
          u & X \\
          0 & h \\
        \end{array}
    \right)=\psi_r(u).
\]
We now describe our main results, first in the global case and then
in the local case.

\subsection{The Global case}\label{globalintro}
Let $F$ be a number field with ad\`ele ring $\A$. We denote by
$Z_G$ the center of $G$. Fix once and for all a unitary character
$\xi$ of $Z_G(F) \bs Z_G(\A)$ and denote by $L^2(Z_G(\A)G(F) \bs
G(\A),\xi)$ the space of functions $\phi$ on $G(F) \bs G(\A)$ such
that $\phi(z g)=\xi(z)\phi(g)$ for all $z \in Z_G(\A),\,g \in G(\A)$
and
\[
\int_{Z_G(\A)G(F) \bs G(\A)}\abs{\phi(g)}^2\ dg <\infty.
\]
We have an orthogonal decomposition
\[
    L^2(Z_G(\A)G(F) \bs G(\A),\xi)=L^2_{disc}(G,\xi) \oplus L^2_{cont}(G,\xi)
\]
of the automorphic spectrum into a discrete part and a continuous
part. The discrete part decomposes further as a direct sum of
irreducible representations, each appearing with multiplicity one.
We say that $\pi$ is a discrete spectrum automorphic representation
of $G(\A)$ with central character $\xi$ if it embeds into
$L^2_{disc}(G,\xi)$ and we refer to this embedding as the
automorphic realization of $\pi$. In \cite{MR91b:22028}, M{\oe}glin
and Waldspurger show that the irreducible components of
$L^2_{disc}(G,\xi)$ are precisely the representations $L(\sigma,t)$
parameterized by pairs $(\sigma,t)$ where $n=rt$ and $\sigma$ is a
cuspidal automorphic representation of $G_r(\A)$ with central
character appropriately related to $\xi$ (see the first paragraph of
\S\ref{global} for notation and the precise statement).

Let $H$ be an algebraic subgroup of $G$ and let $\chi$ be a
character of $H(F)\bs H(\A)$. For an automorphic form $\phi$,
whenever the integral converges, we define
\[
    l_{H}^\chi (\phi)=\int_{H(F) \bs
    H(\A)}\phi(h)\chi(h)dh.
\]
When $\chi$ is the trivial character we will also write $l_H$ for
$l_H^\chi$.
\begin{definition}
We say that an automorphic representation $\pi$ is $(H,\chi)-$
\emph{distinguished} if there is an automorphic form $\phi$ in the
space of $\pi$ such that $l_H^\chi(\phi)\not=0$. When $\chi$ is the
trivial character we will then say that $\pi$ is $H-$distinguished.
\end{definition}
Our main global result is expressed in terms of the classification
of the discrete spectrum as follows.
\begin{theorem}\label{globalmainintro}
Let $\pi$ be an irreducible, discrete spectrum automorphic
representation of $G(\A)$ (of central character $\xi$). Then, there
exists an integer $k$, $0 \le k \le [\frac n2]$ such that $\pi$ is
$(H_{n-2k,2k},\psi_{n-2k})-$distinguished. More precisely, if
$\pi=L(\sigma,t)$ and
\begin{equation}\label{def:k}
\kappa(\pi)=r [\frac t2]
\end{equation}
then $\pi$ is
$(H_{n-2\kappa(\pi),2\kappa(\pi)},\psi_{n-2\kappa(\pi)})-$distinguished.
\end{theorem}


Discrete spectrum automorphic representations are realized as
multi-residues of Eisenstein series. Our proof of Theorem
\ref{globalmain} (that also implies Theorem \ref{globalmainintro})
relies on formula \eqref{periodofres} that expresses the mixed
period $l_{H_{n-2k,2k}}^{\psi_{n-2k}}$ of the multi-residue of an
Eisenstein series in terms of Whittaker and purely symplectic
periods.
\subsection{The local case}\label{locint}
Let $F$ be a $p$-adic field. We will consider only smooth
representations of $G(F)$. In particular, when we say that the
representation $\pi$ of $G(F)$ is unitary we really mean that
$\pi$ is a smooth representation that has a unitary structure.

Let $H$ be an algebraic subgroup of $G$ and let $\chi$ be a
character of $H(F)$.
\begin{definition}
We say that a representation $\pi$ of $G(F)$ is $(H,\chi)-$
\emph{distinguished} if $\Hom_{H(F)}(\pi,\chi)\not=0$. If $\chi$ is
the trivial character we also say that $\pi$ is $H-$distinguished.
\end{definition}

\begin{theorem}\label{localmainintro}
Let $\pi$ be an irreducible, unitary representation of $G(F)$. There
exists an integer $k$, $0 \le k \le [\frac n2]$ such that $\pi$ is
$(H_{n-2k,2k},\psi_{n-2k})-$ distinguished.
\end{theorem}
As in the global case, following our main local result, Theorem
\ref{thm: main local}, we construct a map $\pi \mapsto
\kappa(\pi)$ that assigns (in particular) to any irreducible,
unitary representation an integer $k=\kappa(\pi)$ for which
Theorem \ref{localmainintro} holds. This map is described
explicitly in \S \ref{local} in terms of Tadic's classification of
the unitary dual of $G(F)$ obtained in \cite{MR870688}. Our proof
is local. It is based, however, on the hereditary property of
Whittaker models with respect to parabolic induction and on our
results on purely symplectic models in \cite{mix}. We remark that
our proof in \cite{mix} uses a global argument. Thus, our entire
proof is based on the global theory of automorphic forms. It will
be interesting to see a purely local proof of Theorem
\ref{localmainintro}.
\subsection{Some background on the study of mixed periods}
Theorem \ref{localmainintro} can be interpreted as an existence
statement of certain mixed models. For a decomposition $n=r+2k$ we
introduce the Klyachko model
\[
\mix_{r,2k}=\mbox{Ind}_{H_{r,2k}(F)}^{G(F)}(\psi_{r}).
\]
Note that $\mix_{n,0}$ is the Whittaker model while when $n$ is even
$\mix_{0,n}$ is a purely symplectic model. By Frobenious
reciprocity, for any admissible representation $\pi$ of $G(F)$ we
have
\[
    \Hom_{G(F)}(\pi,\mix_{r,2k})=\Hom_{H_{r,2k}(F)}(\pi,\psi_r).
\]
Thus, if $\pi$ is irreducible then it is
$(H_{r,2k},\psi_r)-$distinguished if and only if it can be realized
in the space $\mix_{r,2k}$. The Klyachko model $\mix_{r,2k}$ is a
mixed (Whittaker-symplectic) model for $G(F)$ and whenever $\pi$ is
$(H_{r,2k},\psi_r)-$distinguished we say that $\pi$ admits the model
$\mix_{r,2k}$.

The models $\mix_{r,2k}$ were first considered by Klyachko in
\cite{kly} in the case where $F$ is a finite field. If $F=\ff_q$ is
the field of $q$ elements, the work of Klyachko suggests that
\[
\mix=\oplus_{k=0}^{[\frac
    n2]}\mix_{n-2k,2k}
\]
is a Gelfand model for $G(\ff_q)$, i.e. it is the direct sum of all
irreducible representations of $G(\ff_q)$ each appearing with
multiplicity one. In other words, for any irreducible representation
$\pi$ of $G(\ff_q)$ we have $m_\pi=1$ where $m_\pi$ is defined by
\[
    m_\pi=\dim_\cc(\Hom_{G(F)}(\pi,\mix)).
\]
As already pointed out by Inglis and Saxl the proof of Klyachko
contains several inaccuracies and gaps and is therefore incomplete.
In \cite{MR1129515} a complete proof is given, using different
methods. The result has applications to the representation theory of
$G(\ff_q)$ and was used for example in \cite{fulman-2003} and in
\cite{MR2206367}. In \cite{thiem-2005}, an analog is proved for the
finite unitary group.

The fact that $m_\pi=1$ for an irreducible representation $\pi$
consists of the following 3 properties: \emph{existence} ($\pi$
admits some Klyachko model), \emph{disjointness} ($\pi$ admits at
most one Klyachko model) and \emph{uniqueness} ($\pi$ imbeds into a
given Klyachko model with at most multiplicity one). Klyachko models
over a $p$-adic field were first studied by Heumos and Rallis in
\cite{MR91k:22036}. They observed that, already when $n=3$ there
exists an irreducible, admissible representation $\pi$ of $G(F)$
that admits no Klyachko model, i.e. such that $m_\pi=0$. However,
when $n \le 4$ they showed that every irreducible, \emph{unitary}
representation $\pi$ of $G(F)$ admits a Klyachko model, i.e. that
$m_\pi \ge 1.$ In general, they also showed the uniqueness of the
purely symplectic model \cite[Theorem 2.4.2]{MR91k:22036}, i.e. that
if $n$ is even and $\pi$ is an irreducible admissible representation
of $G(F)$ then
\begin{equation}\label{eq: symp uniq}
    \dim_\cc(\Hom_{G(F)}(\pi,\mix_{0,n}))\le 1.
\end{equation}
Another result claimed in \cite[Theorem 3.1]{MR91k:22036} is
disjointness of Klyachko models for irreducible, unitary
representations. Unfortunately, the proof is based on
\cite[Proposition 1.3]{kly} which is false. To be more precise, the
proof given in \cite[\S 3]{MR91k:22036} could have been based on the
statement in \cite[\S 1.1]{kly}, which is a weaker statement then
\cite[Proposition 1.3]{kly} and which may still be true but, to our
knowledge, has not yet been proved. We will obtain the disjointness
and uniquness of Klyachko models over a $p$-adic field in an
upcoming paper. The local part of the current paper treats the
existence of Klyachko models.

In \cite{mix}, we provided a family of unitary representations of
$G(F)$ that admit a purely symplectic model. From Theorem
\ref{localmainintro} we get that $m_\pi \ge 1$ for any irreducible
unitary representation $\pi$ of $G(F)$. We also promised in
\cite{mix} that the current work will characterize, in particular,
all irreducible unitary representations admitting a symplectic
model. This was based on the unitary disjointness that, we only
later observed, remains unproved. We will therefore only deliver our
promise in our upcoming paper when we prove disjointness of the
Klyachko models. It will also be interesting to study the
analogous problem in the archimedean case, and the global mixed
periods for the continuous automorphic spectrum of $G$. Langlands,
described the continuous spectrum in terms of discrete spectrum
datum. Using this description, we believe that a global analogue
of Theorem \ref{localmainintro}, properly formulated, should be
true and we also hope to address this problem in the future. Our
study of the global mixed periods was motivated by its analogy
with the local problem. This analogy, was already suggested by
Heumos in his survey paper on the subject \cite{{MR1216185}}.

\begin{remark}
The focus of this paper is on non vanishing of periods. Thus, the
way Haar measures are normalized plays no role in the proofs of the
main results. We therefore do not choose a specific normalization
for the Haar measures appearing in this work and (with the exception
of \S \ref{Lfunc}) will make no further comments regarding the
choice of measures.
\end{remark}

\subsection{Acknowledgement}
We are grateful to Joseph Bernstein for his help and patience. His
insistence for conceptual arguments significantly simplified this
work.

\section{Global mixed periods}\label{global}

M{\oe}glin and Waldspurger classified the discrete spectrum
automorphic representations of $G(\A)$ in \cite{MR91b:22028}. We
recall their result.

Let $P=MU$ be a standard parabolic subgroup of $G$ with Levi subgroup $M$ and unipotent radical $U$.
If $P$ is of type $(n_1,\dots,n_t)$ then for an automorphic representation
$\tau=\tau_1 \otimes \cdots \otimes \tau_t$ of $M(\A)$ and for
$\lambda=(\lambda_1,\dots,\lambda_t) \in \cc^t$ let $\tau[\lambda]$
be the representation on the space of $\tau$ defined by
\[
    \tau[\lambda]=\abs{\det }^{\lambda_1}\tau_1 \otimes \cdots
    \otimes \abs{\det }^{\lambda_t}\tau_t
\]
and let $I(\tau,\lambda)=I_P^G(\tau,\lambda)$ be the representation
of $G(\A)$ parabolically induced from $\tau[\lambda]$. For a
positive integer  $t$ let
\begin{equation}\label{lambda}
    \Lambda_t=(\frac{t-1}2,\frac{t-3}2,\dots,\frac{1-t}2).
\end{equation}
Fix a character $\xi$ of $F^\times \bs \A^\times$ and identify $Z_G(\A)$ with $\A^\times$.
Let $rt=n$ and let $\sigma$ be an irreducible,
cuspidal automorphic representation of $G_r(\A)$. Assume that $P$ is
of type $(r,\dots,r)$ and let $\tau=\sigma^{\otimes t}$. The
representation $I(\tau,\Lambda_t)$ has a unique irreducible
quotient, which we denote by $L(\sigma,t)$. Note that the central character of
$I(\tau,\Lambda_t)$ and therefore also of  $L(\sigma,t)$ is
$\omega_\sigma^t$ where $\omega_\sigma$ is the central character of
$\sigma$.
\begin{theorem}[M{\oe}glin-Waldspurger]\label{MW}
Let $n=rt$ and let $\sigma$ be an irreducible, cuspidal automorphic
representation of $G_r(\A)$ so that $\omega_\sigma^t=\xi$. The
representation $L(\sigma,t)$ occurs in $L_{disc}^2(G,\xi)$ with
multiplicity one and every irreducible component of
$L_{disc}^2(G,\xi)$ is of the form $L(\sigma,t)$ for such a pair
$(\sigma,t).$
\end{theorem}
We now turn to the proof of Theorem \ref{globalintro}. Let
$\pi=L(\sigma,t)$ be an irreducible component of
$L_{disc}^2(G,\xi)$. If $\pi$ is cuspidal (i.e. if $t=1$), it is
well known that the Whittaker functional $l_{U_n}^{\psi_n}$ is not
identically zero on $\pi$, in other words $\pi$ is
$(U_n,\psi_n)-$distinguished. On the other hand, in \cite[Theorem
3]{Odist} we show that $\pi$ is $Sp(n)-$distinguished if and only if
$t$ is even. Theorem \ref{globalmainintro} therefore follows from
\begin{theorem}\label{globalmain}
Let $n=(2m+1)r$ and let $\sigma$ be an irreducible, cuspidal
automorphic representation of $G_r(\A)$. The representation
$L(\sigma,2m+1)$ is $(H_{r,2mr},\psi_r)-$distinguished.
\end{theorem}

The proof of Theorem \ref{globalmain} is given in \S\ref{global
proof}. We start with some notation and a summary of necessary
facts.

\subsection{Eisenstein series, intertwining operators and multi-residues}\label{Eis}

When we say that $P=MU$ is a standard parabolic subgroup of $G$ with
its standard Levi decomposition, we mean that $M$ is the standard
Levi subgroup and $U$ is the unipotent radical of $P$. Throughout,
$P=MU$ and $Q=LV$ will denote standard parabolic subgroups of $G$
with their standard Levi decompositions so that $P$ is contained in
$Q$.

For integers $a \le b$ let $[a,b]=\{a,a+1,\dots,b\}$. We identify
the index set $\Delta=[1,n-1]$ with the set of simple roots of $G$
with respect to the standard Borel subgroup of upper triangular
matrices in $G$ and let $\Delta^M$ denote the set of those indices
$i$ that correspond to roots in $M$. If $P$ is of type
$(n_1,\dots,n_t)$ then
\[
    \Delta^M=\Delta \setminus
    \{n_1,n_1+n_2,\dots,n_1+\cdots+n_{t-1}\}.
\]

Let $S_n$ denote the group of permutations on $[1,n]$. We denote by
$W_M$ the Weyl group of $M$ and let $W=W_G$. We identify $W$ with
$S_n$ and when convenient we consider an element of $S_n$ as a
permutation matrix in $G$. Each double coset in $W_{L} \bs W/W_{M}$
contains a unique representative of minimal length - the left $W_{L}$
and right $W_{M}$ reduced representative. We denote by ${}_{L}
W_{M}$ the set of reduced representatives of the double cosets and
let
\[
    {}_{L}
W_{M}^c=\{w \in {}_{L} W_{M}: wMw^{-1} \subset L\}.
\]
If $n=tr$ and $P$ is of type $(r,\dots,r)$ the set ${}_M W_M^c$
consists of the Weyl elements permuting the blocks of $M$ and we
identify ${}_M W_M^c$ with $S_t$.

For an algebraic group $Y$ defined over $F$ let $X^*(Y)$ be the
lattice of rational characters of $Y$, let $\aaa_Y^*=X^*(Y) \otimes
\rr$ and let $\aaa_Y$ be the dual vector space. Denote by
$\sprod{\cdot}{\cdot}$ the pairing between $\aaa_Y$ and $\aaa_Y^*$.
If $P$ is of type $(n_1,\dots,n_t)$ then we identify $\aaa_M$ and
its dual with $\rr^t$. The pairing between $\aaa_M$ and $\aaa_M^*$
is then the standard inner product on $\rr^t$. There is a natural
embedding of $\aaa_L$ into $\aaa_M$ with an orthogonal decomposition
\[
    \aaa_M=\aaa_L \oplus \aaa_M^L.
\]
Denote by $X \mapsto X_L$ the orthogonal projection from $\aaa_M$ to
$\aaa_L$. We use similar notation for the corresponding
decomposition in the dual space and denote by $\lambda \mapsto
\lambda_L$ the orthogonal projection of $\aaa_M^*$ to $\aaa_L^*$.
Using the Iwasawa decomposition $G(\A)=U(\A)M(\A)K$, where $K$ is
the standard maximal compact subgroup of $G(\A)$, we define the
height function $H_M:G(\A) \to \aaa_M$ by the requirement that for
every $\chi \in X^*(M)$ we have
\[
    e^{\sprod{\chi}{H(umk)}}=\prod_v \abs{\chi_v(m_v)}_v.
\]
The map $H_M$ induces an isomorphism $M(\A)/M(\A)^1 \to \aaa_M$. Let
$\rho_P=(\Lambda_n)_M \in (\aaa_M)^*$ where $\Lambda_n$ is given by
\eqref{lambda}. Thus,
\[
p \mapsto e^{\sprod{2\rho_P}{H_M(p)}}
\]
is the modulus function of $P (\A)$.

Let $\tau$ be an irreducible, cuspidal
automorphic representation of $M(\A)$. We identify the
representation spaces of $I_{P}^G(\tau,\lambda)$ for all
$\lambda \in \aaa_{M,\cc}^*$ with the representation space $I_P^G(\tau)$ with $\lambda=0$. In
particular, we have
\[
    \varphi(pg)=e^{\sprod{\rho_P}{H_M(p)}}\tau(m) \varphi(g)
\]
whenever $\varphi \in I_P^G(\tau),\,p=mu \in P(\A),\,m \in M(\A),\,u \in U(\A)$ and $g \in G(\A)$.
For $\varphi
\in I_P^G(\tau)$ and $\lambda \in \aaa_{M,\cc}^*$ we denote
by $\varphi_\lambda$ the standard holomorphic section given by
\[
    \varphi_\lambda(g)=\varphi(g)e^{\sprod{\lambda}{H_M(g)}}.
\]
The action of the representation $I_P^G(\tau,\lambda)$ on
the space $I_P^G(\tau)$ is then given by
\[
    (I_P^G(g,\tau,\lambda)\varphi)_\lambda (x)=\varphi_\lambda(xg)
\]
for $g,\,x \in G(\A)$.

Let $\lambda \in \mathfrak{a}_{M,\bbc}^*$ and let $\varphi \in
I_P^G(\tau)$. The Eisenstein series $E^Q(\varphi,\lambda)$ is defined as
the meromorphic continuation of the series
\[
    E^Q(g,\varphi,\lambda)=\sum_{\gamma \in (P\cap L)(F) \bs
    L(F)}\varphi_\lambda(\gamma g).
\]
When $Q=G$ we also set $E(\varphi,\lambda)=E^Q(\varphi,\lambda)$.

Assume now that $n=tr$, that $P$ is of type $(r,\dots,r)$, that $Q$
is of type $(m_1 r,\dots,m_s r)$, that $\sigma$ is an irreducible,
cuspidal automorphic representation of $G_r(\A)$ and that
$\tau=\sigma^{\otimes t}$ is the corresponding cuspidal
representation of $M(\A)$. For any $w \in {}_MW_M^c=S_{t}$ we let
$M(w,\lambda):I(\tau,\lambda) \to I(\tau,w\lambda)$ be the standard
(non normalized) intertwining operator. It is defined by the
meromorphic continuation of the integral
\begin{equation}\label{intop}
    (M(w,\lambda)\varphi)_{w\lambda}(g)=\int_{(U \cap wUw^{-1})(\A)\bs
    U(\A)}\varphi_{\lambda}(w^{-1}ug)du.
\end{equation}
Its domain of convergence includes that of the Eisenstein series and
it admits a meromorphic continuation. Let
\[
\Lambda^Q=(\Lambda_{m_1},\dots,\Lambda_{m_s}) \in (\aaa_M^L)^*
\]
and let $\mu \in \aaa_L^*$.
For $\varphi \in I_P^G(\tau)$ the expression
\[
E^Q(\varphi,\lambda+\mu)\prod_{i \in
\Delta^L}(\lambda_i-\lambda_{i+1}-1)
\]
is holomorphic at $\lambda=\Lambda^Q$. We may therefore define the
multi-residue of the Eisenstein series by
\[
    E^Q_{-1}(\varphi,\mu)=\lim_{\lambda \to \Lambda^Q}E^L(\varphi,\lambda+\mu)\prod_{i \in
    \Delta^L}(\lambda_i-\lambda_{i+1}-1).
\]
It defines a surjective intertwining operator
\[
    E_{-1}^Q(\mu):I_{P}^G(\tau,\Lambda^Q+\mu) \to I_Q^G(L(\sigma,m_1)
    \otimes \cdots \otimes L(\sigma,m_t),\mu).
\]
When $Q=G$ we also denote the multi-residue operator by
$E_{-1}=E_{-1}^G(0)$. It is then a surjective intertwining operator
from $I(\tau,\Lambda_t)$ to $L(\sigma,t)$ that realizes the
representation $L(\sigma,t)$ in the space $L^2_{disc}(G,\xi)$ of
automorphic forms. We also consider the multi-residue of the
intertwining operator $M(w,\lambda)$. Let
\[
    \Delta(w)=\{i \in [1,t-1]: w(i)>w(i+1)\}
\]
and set
\begin{equation}\label{res int op}
    M_{-1}(w)=\lim_{\lambda \to \Lambda_t} M(w,\lambda)\prod_{i
    \in \Delta(w)}(\lambda_i-\lambda_{i+1}+1).
\end{equation}
It is an intertwining operator from $I(\tau,\Lambda_t)$ to
$I(\tau,w\Lambda_t) $.
For an automorphic form $\phi$ on $G(\A)$ we define its constant term along $Q$
by
\[
        \phi_Q(g)=\int_{V(F) \bs V(\A)} \phi(vg)dv.
\]
The function $\ell \mapsto \phi_Q(\ell)$ is an automorphic form on $L(\A)$.
We denote it by $\phi_Q[e]$.
\subsection{Proof of Theorem \ref{globalmain}}\label{global proof}

Fix a decomposition $n=(2m+1)r$. Let $H=H_{r,2mr}$, let $P$ be of
type $(r,\dots,r)$ and let $Q=LV$ be of type $(r,2mr)$ (i.e. $Q$ is
the standard maximal parabolic subgroup of $G$ containing $H$). Note
that
\[
H \simeq (U_r \times Sp(2mr))V \text{ and } L \simeq G_r \times
G_{2mr}.
\]
If $\phi$ is an automorphic form in the discrete spectrum then it is
easy to see that
\begin{equation}\label{period const term}
    l_H^{\psi_r}(\phi)=(l_{U_r}^{\psi_r} \otimes l_{Sp(2mr)})  (\phi_Q[e]).
\end{equation}
Let $\sigma$ be a cuspidal automorphic representation of $G_r(\A)$,
let $\tau=\sigma^{\otimes 2m+1}$ and let $\pi=L(\sigma,2m+1)$. Our
goal is to show that the mixed period $l_H^{\psi_r}$ is not
identically zero on $\pi$.
As already explained, it is known that
$l_{U_r}^{\psi_r} \otimes l_{Sp(2mr)}$ is not zero on
$\sigma[-m] \otimes  L(\sigma,2m)[\frac12]$.
From \eqref{period const term} we see that it is enough to show the following.
\begin{proposition}\label{prop: global onto}
The map $\phi \mapsto \phi_Q[e]$ defines a surjection from $\pi$
to $\sigma[-m] \otimes  L(\sigma,2m)[\frac12].$
\end{proposition}
\begin{proof}
Note first that $\sigma[-m] \otimes  L(\sigma,2m)[\frac12]$ is irreducible
and therefore it is enough to prove that the map $\phi \mapsto \phi_Q[e]$ is not identically
zero on $\pi$ and that its image indeed lies in
\[
\sigma[-m] \otimes  L(\sigma,2m)[\frac12].
\]
Since $\phi_Q(\ell g)$ is the value at $\ell$ of $(\pi(g)\phi)_Q[e]$ for $\ell \in L(\A),\,g \in G(\A)$, to prove that the map is not zero it is enough to show that the constant term map $\phi \mapsto \phi_Q$ is not identically zero on $\pi$. But it is well known (e.g. \cite{MR85k:22045}) that $\phi \mapsto \phi_P$ is not identically zero (in fact $\phi \mapsto \phi_P$ defines an imbedding of $L(\sigma,2m+1)$ in $I(\tau,-\Lambda_{2m+1})$) and we  have
\[
    \phi_P(g)=\int_{(U \cap L)(F) \bs (U\cap L)(\A)}\phi_Q(ug)\ du.
\]
It therefore only remains to show that $\phi_Q[e]$ lies in the space of the automorphic representation
$\sigma[-m] \otimes  L(\sigma,2m)[\frac12]$ of $L(\A).$
To see this we use the automorphic realization of $L(\sigma,2m+1)$ and compute
the constant term of multi-residues of Eisenstein series.
Denote by $w_Q$ the longest element in ${}_L W_M.$ Thus, $w_Q \in
{}_M W_M^c$ and as a permutation in $S_{2m+1}$ it is the cycle
$(1,2,\dots,2m+1)$.
Let
\[
\mu_Q=w_Q\Lambda_{2m+1}-\Lambda^Q=(-m,\frac12) \in \aaa_L^*.
\]
The formula for the constant term that we obtain in Lemma \ref{const
term}, implies that $E_{-1}(\varphi)_Q$ lies in the image of the
operator
\[
E_{-1}^Q(\mu_Q):I(\tau,\Lambda_{2m+1}) \to I_Q^G(\sigma \otimes L(\sigma,2m),\mu_Q)
\]
for $E_{-1}(\varphi)$ in the space of $\pi$
and therefore that $E_{-1}(\varphi)_Q[e]$ lies in $\sigma[-m] \otimes L(\sigma,2m)[\frac12]$. It remains only to compute the constant term.
\begin{lemma} \label{const term}
For every $\varphi \in I(\tau)$ we have
\[
    E_{-1}(\varphi)_Q=
    E_{-1}^Q ((M_{-1}(w_Q) \varphi),\mu_Q).
\]
\end{lemma}
\begin{proof}
For $\varphi \in I(\tau)$ the constant term of the Eisenstein series
$E(\varphi,\lambda)$ is given by
\begin{equation}\label{constofeis}
    E(\varphi,\lambda)_Q=\sum_{w \in {}_L W_M^c}E^Q(M(w,\lambda)\varphi,w\lambda).
\end{equation}
The multi-residue operator $\lim_{\lambda \to \Lambda_{2m+1}}
\prod_{i=1}^{2m}(\lambda_i-\lambda_{i+1}-1) $ and the constant term
operator are interchangeable. We show that after applying the
multi-residue operator to (\ref{constofeis}), only the term
associated with $w_Q$ survives. The map $w \mapsto w^{-1}(1)$ is a
bijection from ${}_L W_M^c$ to $[1,2m+1]$. Let $w^{(i)} \in {}_L
W_M^c$ be such that $w^{(i)}(i)=1$, thus $w^{(i)}=(1,2,\dots,i)$
(and in particular $w^{(2m+1)}=w_Q$). For the term associated to the
identity element $w^{(1)}$, note that $E^Q(\varphi,\lambda)
\prod_{i=2}^{2m}(\lambda_i-\lambda_{i+1}-1)$ is holomorphic at
$\Lambda_{2m+1}$. Therefore, the contribution to (\ref{constofeis})
of the term associated to the identity Weyl element vanishes after
taking the multi-residue operator. For $i>1$ we have $
\Delta(w^{(i)})=\{i-1\}$ and therefore
$(\lambda_{i-1}-\lambda_i-1)M(w^{(i)},\lambda)$ is holomorphic at
$\Lambda_{2m+1}$. Note also that
\begin{multline*}
    \{(w^{(i)})^{-1}(j):j  \in
    [1,2m] \text{ and }(w^{(i)}\Lambda_{2m+1})_j-(w^{(i)}\Lambda_{2m+1})_{j+1}=1\}\\=[1,2m]\setminus
    \{i-1,i\}
\end{multline*}
and therefore that
$$
    E^Q(M(w^{(i)},\lambda)\varphi,w^{(i)}\lambda)\prod_{j\in
    [1,2m]\setminus \{i\}}(\lambda_j-\lambda_{j+1}-1)
$$
is holomorphic at $\Lambda_{2m+1}$. Thus if $i<2m+1$ the $w^{(i)}$
contribution to (\ref{constofeis}) vanishes after applying the
multi-residue operator. It follows that
\[
E_{-1}(\varphi)_Q=\lim_{\lambda \to
\Lambda_{2m+1}}E^Q(M(w_Q,\lambda)\varphi,w_Q\lambda)\prod_{i=1}^{2m}(\lambda_i-\lambda_{i+1}-1).
\]
The lemma follows.
\end{proof}
This completes the proof of Proposition \ref{prop: global onto}.
\end{proof}
As we already explained, this completes the proof of Theorem
\ref{globalmain}.

It follows from Lemma
\ref{const term} that we can express the mixed period as
\begin{equation}\label{periodofres}
    l_H^{\psi_r}(E_{-1}(\varphi))=(l_{U_r}^{\psi_r} \otimes l_{Sp(2mr)})
    (E_{-1}^Q (M_{-1}(w_Q \,\varphi),\mu_Q)[e])
\end{equation}
where $f[e] \in \sigma[-m] \otimes L(\sigma,2m)[\frac12]$ is the valuation at $e$ of
an element $f \in I_Q^G((\sigma \otimes L(\sigma,2m),\mu_Q).$

\begin{corollary}\label{factorizable}
The mixed period integral
$l_{H_{n-2\kappa(\pi),2\kappa(\pi)}}^{\psi_{n-2\kappa(\pi)}}$ is factorizable on
the discrete spectrum representation $\pi$.
\end{corollary}
\begin{proof}
It is well known that on a cuspidal representation, the Whittaker
functional is factorizable. It also follows from the explicit
formula in \cite[Theorem 1.1]{OII}, that the purely symplectic
period is factorizable on the residual spectrum. It then follows
from \eqref{periodofres} that the mixed period
$l_{H_{r,2mr}}^{\psi_r}$ is factorizable on $L(\sigma,2m+1)$ for
$\sigma$ a cuspidal representation of $G_r(\A)$.
\end{proof}
\begin{remark}
Note that the factorization of the period is obtained using a
formula for the mixed period despite the fact that local
multiplicity one for the mixed models is not yet known. Note also
that based on the conjectural disjointness of models for unitary
representations we expect the period $l_{H_{r,2k}}^{\psi}$ to vanish
on every discrete spectrum representation $\pi$ such that
$\kappa(\pi) \not=k$.
\end{remark}
\subsection{Some explicit formulas for the periods}\label{Lfunc}
Formula \eqref{periodofres} indicates that the mixed period is
related to special values of the Rankin-Selberg $L$-function
associated to $\sigma$ and its contragradient $\tilde \sigma$. Let
$L_\sigma(s)=L(s,\sigma \times \tilde\sigma)$ and fix once and for
all a finite set of places $S$ containing the infinite places and
such that for $v \not\in S$ the conductor of $\psi_v$ is $\oo_v$. If
$\sigma$ is an everywhere unramified cuspidal representation of
$G(\A)$ and $\phi_0$ is its $L^2$-normalized spherical vector, it
may be that $l_{U_n}^{\psi_n}(\phi_0)$ equals zero, but in any case
we can write $l_{U_n}^{\psi_n}(\sigma(x)\phi_0)=\prod_v
W^{\psi_v}_v(x)$ where $W^{\psi_v}_v$ is the spherical Whittaker
function for all $v$, and it is normalized so that
$W^{\psi_v}_v(e)=1$ for all $v \not\in S$. Clearly, there exists $g
\in G(\A)$ such that $l_{U_n}^{\psi_n}(\sigma(g)\phi_0)\not=0$ (in
fact, we may and do choose $g$ such that $g_v=e$ for $v \not\in S$).
We denote by $W^{\psi_v}_{1,v}$ the $L^2$-normalized spherical
Whittaker function and refer to \cite[\S2.2]{cup} for the
normalization of $W_{1,v}^{\psi_v}$ and for more details. It follows
from Jacquet's formula for the inner product of cusp forms that
\[
    \prod_{v \in
    S}\abs{\frac{W^{\psi_v}_v(g)}{W^{\psi_v}_{1,v}(g)}}^2=\frac1{\res_{s=1}L_\sigma^S(s)}
\]
where $L_\sigma^S(s)$ is the partial $L$-function away from $S$. We
then have
\begin{equation}\label{whittaker}
    \abs{l_{U_n}^{\psi}(\pi(g)\phi_0)}^2=\frac{\prod_{v \in
    S}\alpha_v(\sigma_v;g_v)}{\res_{s=1}L_\sigma(s)}
\end{equation}
where
\[
    \alpha_v(\sigma_v;g_v)=L_{\sigma_v}(1)\abs{W^{\psi_v}_{1,v}(g_{v})}^2.
\]
Assume now that $\sigma$ is an everywhere unramified cuspidal
representation of $G_r(\A)$ and let $\phi_0$ be the
$L^2$-normalized, spherical element of the discrete spectrum
representation $\pi=L(\sigma,t)$. We have the following formula for
the mixed period of $\phi_0$.
\begin{proposition}
For a certain normalization of Haar measures independent of $\sigma$
we have that if $t=2m$ is even then
\[
\abs{l_{Sp(n)}\phi_0}^2= \frac{L_\sigma(2)L_\sigma(4)\cdots
L_\sigma(2m)} {\res_{s=1}L_\sigma(s)L_\sigma(3)\cdots
L_\sigma(2m-1)}
\]
and if $t=2m+1$ is odd then there exists an element $g_0 \in
G_r(\A)$ such that
\[
\abs{l_{H_{r,2mr}}^\psi(\pi(\diag(g_0,1_{2mr}))\phi_0)}^2=\frac{\prod_{v
\in
S}\alpha_v(\sigma_v;g_{0,v})}{\res_{s=1}L_\sigma(s)}\prod_{j=1}^m
\frac{L_\sigma(2j)}{L_\sigma(2j+1)}.
\]
\end{proposition}
\begin{proof}
When $t$ is even, the result is \cite[Theorem 4]{Odist} and when
$t=1$ the formula is \eqref{whittaker}. Let $v_0$ be the
$L^2$-normalized spherical cusp form in the space of $\sigma$ and
let $\varphi_0^{(t)}$ be the spherical section in
$I(\tau,\Lambda_t)$ normalized so that
$\varphi_0^{(t)}(e)=v_0^{\otimes t}$. We now prove the formula for
$t=2m+1>1$ using \eqref{periodofres} and a computation similar to
that in \cite{Odist}. Indeed, we have
\[
    \phi_0=\frac{E_{-1}(\varphi_0^{(t)})}{\|E_{-1}(\varphi_0^{(t)})\|_2}.
\]
As explained in \cite{Odist}, Langlands showed that
\[
    \|E_{-1}(\varphi_0^{(t)})\|_2^{-2}=
    \frac{L_\sigma(2)L_\sigma(3) \cdots L_\sigma(t)}{(\res_{s=1}L_\sigma(s))^{t-1}}
\]
(note the typo in the proof of \cite[Theorem 4]{Odist} where
$\|E_{-1}(\varphi_0)\|_2^{2}$ should be replaced by
$\|E_{-1}(\varphi_0)\|_2^{-2}$). On the other hand
\[
    M_{-1}(w_Q)\varphi_0^{(2m+1)}=\frac{\res_{s=1}L_\sigma(s)}{L_\sigma(2m+1)}\varphi_0^{(2m+1)}
\]
and therefore
\[
    E_{-1}^Q (M_{-1}(w_Q)\varphi_0^{(2m+1)},\mu_Q)[e]=
    \frac{\res_{s=1}L_\sigma(s)}{L_\sigma(2m+1)}(v_0 \otimes
    E_{-1}^{G_{2mr}}(\varphi_0^{(2m)})).
\]
The computation in the proof of \cite[Theorem 4]{Odist} gives
\[
    l_{Sp(2mr)}(E_{-1}^{G_{2mr}}(\varphi_0^{(2m)}))=
    \frac{(\res_{s=1}L_\sigma(s))^{m-1}}{L_\sigma(3)L_\sigma(5)\cdots
    L_\sigma(2m-1)}.
\]
Plugging all this to \eqref{periodofres} we get that
\[
    \abs{l_{H_{r,2mr}}^{\psi_r}(\phi_0)}^2=\abs{l_{U_r}^{\psi_r}(v_0)}^2
    \frac{L_\sigma(2)L_\sigma(4)\cdots L_\sigma(2m)}{L_\sigma(3)L_\sigma(5)\cdots
    L_\sigma(2m+1)}.
\]
As already explained, for some everywhere unramified cuspidal
representations $\sigma$ it may be that $l_{U_r}^{\psi_r}(v_0)=0$,
however, there exists $g_0 \in G_r(\A)$ such that
$l_{U_r}^{\psi_r}(\sigma(g_0)v_0)\not=0$. A similar computation then
gives that for $g=\diag(g_0,1_{2mr})$ we have
\[
    \abs{l_{H_{r,2mr}}^{\psi_r}(\pi(g)\phi_0)}^2=\abs{l_{U_r}^{\psi_r}(\sigma(g_0)v_0)}^2
    \frac{L_\sigma(2)L_\sigma(4)\cdots L_\sigma(2m)}{L_\sigma(3)L_\sigma(5)\cdots
    L_\sigma(2m+1)}.
\]
The formula now follows from \eqref{whittaker}.
\end{proof}
\section{Local mixed models}\label{local}
Let $F$ be a non-archimedean local field of characteristic zero. For
any algebraic group $Y$ defined over $F$ we denote from now on by
$Y$ the group $Y(F)$ of $F$-rational points. Our goal is to prove
Theorem \ref{localmainintro}, i.e. to show that for any irreducible,
unitary representation $\pi$ of $G$ we can attach an integer $\kappa(\pi)
\in [0,[\frac n2]]$ so that
\[
    \Hom_{H_{n-2\kappa(\pi),2\kappa(\pi)}}(\pi,\psi_{n-2\kappa(\pi)}) \not=0.
\]
The index $\kappa(\pi)$ is expressed in terms of the classification
of the unitary dual of $G$ obtained by Tadic in \cite{MR870688}. Our
proof is based on the properties of derivatives for representations
of $G$ introduced by Gelfand-Kazhdan in \cite{MR0404534}. In
\cite[Theorem 6.1]{MR584084}, Zelevinsky classified all irreducible
representation of $G$ in terms of cuspidal representations.
Furthermore, the highest derivative of any irreducible
representation is irreducible. The derivatives are computed by
Zelevinsky in \cite[Theorem 8.1]{MR584084}. This computation is for
derivatives in the `opposite direction' to those suited for the
study of Klyachko models with respect to the pairs
$(H_{r,2k},\psi_r)$. It will be more convenient to apply
Zelevinsky's results as they stand. We therefore begin by
introducing a closely related family of mixed models compatible with
the derivatives computed by Zelevinsky.
\subsection{Another family of mixed models}
The derivatives computed by Zelevinsky are more suited to the study
of mixed models with respect to the pairs $(H_{2k,r}',\psi_r')$
where
\[
    H_{2k,r}'=\{\left(
                \begin{array}{cc}
                  h & X\\
                  0 & u \\
                \end{array}
              \right):h \in Sp(2k),\,u \in U_r,\,X \in M_{2k \times
              r}\}
\]
and
\[
    \psi_r'\left(
                \begin{array}{cc}
                  h & X\\
                  0 & u \\
                \end{array}
              \right)=\psi(u_{1,2}+\cdots+u_{r-1,r})
\]
whenever $u=(u_{i,j}) \in U_r$. The next lemma is relating between
the two families of mixed models.
\begin{lemma}\label{lemma: H to H'}
Let $(\pi,V)$ be a representation of $G$. There is a linear
isomorphism
\[
    \Hom_{H_{r,2k}}(\pi,\psi_r)\simeq \Hom_{H_{2k,r}'}(\tilde\pi,\bar\psi_r').
\]

\end{lemma}
\begin{proof}
Fix $r$ and $2k$ and let $H=H_{r,2k}$ and $H'=H_{2k,r}'$. Let $\tau$
be the involution on $G$ defined by
\[
    \tau(g)=w{}^t g^{-1} w^{-1}
\]
where
\[
    w=\left(
        \begin{array}{cc}
          0 & w_r \\
          1_{2k} & 0 \\
        \end{array}
      \right).
\]
Let $(\pi^\tau,V)$ be the representation of $G$ on $V$ given by
$\pi^\tau(g)v=\pi(\tau(g))v$. It is well known (\cite{MR0404534})
that $\pi^\tau$ is isomorphic to $\tilde \pi$ and therefore that
there is a linear isomorphism
\[
    \Hom_{H'}(\tilde\pi,\bar\psi_r')=\Hom_{H'}(\pi^\tau,\bar\psi_r').
\]
Note further that $\tau(H)=H'$ and that
$\psi_r(h)=\bar\psi_r'(\tau(h)).$ This implies that the identity map
on the space of $\pi$ defines a linear isomorphism
\[
    \Hom_{H}(\pi,\psi_r)=\Hom_{H'}(\pi^\tau,\bar\psi_r').
\]
\end{proof}
\subsection{The dual of $G$}
Zelevinski classified in \cite{MR584084} all irreducible
representations of $G$ in terms of cuspidal representations. For a
representation $\sigma$ of $G_r$ and for $\lambda \in \cc$ let
$\sigma[\lambda]=\abs{\det}^\lambda \sigma.$ If $\sigma_i$ is a
representation of $G_{r_i},\,i=1,\dots,t$ we denote by $\sigma_1
\times \cdots \times \sigma_t$ the representation of
$G_{r_1+\cdots+r_t}$ parabolically induced from $\sigma_1 \otimes
\cdots \otimes \sigma_t$.  Let $\ccc$ denote the collection of all
irreducible, cuspidal representations of $G_r$ for all positive
integers $r$. For any $a,\,b \in \rr$ such that $0 \le b-a \in \zz$
and any $\rho \in \ccc$ the set
\[
    \Delta=[a,b]^{(\rho)}=\{\rho[a+i]:0 \le i \le b-a\}
\]
is called a segment. The representation $\rho[a] \times \rho[a+1]
\times \cdots \times \rho[b]$ has a unique irreducible
subrepresentation which is denoted by $\langle \Delta \rangle$. We
say that a segment $\Delta=[a,b]^{(\rho)}$ precedes the segment
$\Delta'=[a',b']^{(\rho)}$ if $\Delta' \not\subset \Delta$,
$\rho'[a']=\rho[a+k]$ for some integer $k>0$ and $\Delta \cup
\Delta'$ is also a segment. Denote by $\oo$ the collection of all
multisets of segments $[a,b]^{(\rho)}$ with $\rho \in \ccc$ and
$b-a$ a non negative integer. For any $a \in \oo$ the segments in
$a$ can be arranged as $a=\{\Delta_1,\dots,\Delta_t\}$ so that for
all $1 \le i<j \le t$ the segment $\Delta_i$ does not precede
$\Delta_j.$ In this case the representation $\langle \Delta_1
\rangle \times \cdots \times \langle \Delta_t \rangle$ has a unique
irreducible subrepresentation denoted by $\langle a \rangle$. This
is the statement of \cite[Theorem 6.1 (a)]{MR584084}. The rest of
\cite[Theorem 6.1]{MR584084} is the following statement.
\begin{theorem}[Zelevinsky]
Any irreducible representation of $G$ is of the form $\langle
a\rangle$ for some $a \in \oo$ uniquely determined by the multiset
$a$.
\end{theorem}
\subsection{Derivatives of representations of $G$}
Derivatives of representations of $G$ were introduced in
\cite{MR0404534}. For a non negative integer $k \le n$ the $k$th
derivative is a functor taking a representation $\pi$ of $G$ to a
representation $\pi^{(k)}$ of $G_{n-k}$. It is defined as follows.
Let $P_k$ be the mirabolic subgroup of matrices in $G_k$ with last
row $(0,\dots,0,1)$ and let $V_k$ be its unipotent radical. We imbed
$G_{k-1}$ in the upper left block of $G_k$ whenever convenient. The
functors $\Phi^-$ from representations of $P_k$ to representations
of $P_{k-1}$ and
 $\Psi^-$ from representations of $P_k$ to representations of $G_{k-1}$ are defined in \cite[\S 3.2]{MR58:28310}.
For a representation $\tau$ of $P_k$, $\Phi^-(\tau)$ is the
normalized Jacquet functor of $\tau$ with respect to $V_k$ and the
character $\theta(v)=\psi(v_{k-1,k}),\,v=(v_{i,j} )\in V_k$ regarded
as a representation of $P_{k-1}$ imbedded in $P_k$ and
$\Psi^-(\tau)$ is the normalized Jacquet module of $\tau$ with
respect to $V_k$ and the trivial character regarded as a
representation of $G_{k-1}$ imbedded in $P_k$. If $\Phi^{- (m)}$
denotes the functor $\Phi^- $ applied $m$ times (hence it is a
functor from representations of $P_n$ to representations of
$P_{n-m}$) and $\pi$ is a representation of $G$ then the $r$th
derivative of $\pi$ is given by
\[
\pi^{(r)}=\Psi^- \Phi^{-(r)}(\pi_{|P_n}).
\]
The main property of derivatives relevant to the study of mixed
models lies in the content of \cite[Proposition 3.7]{MR584084}. We
now recall this property. For any representation $(\pi,V)$ of $G_n$
let $(\pi^{(r)},V^{(r)})$ be the $r$th derivative. Then there is a
surjective morphism $A=A^{(r)}_\psi(\pi):V \to V^{(r)}$ so that
\begin{equation}\label{derivative property}
A(\pi
\left(
  \begin{array}{cc}
    g & X \\
    0 & u \\
  \end{array}
\right) v)=\psi_r(u)\pi^{(r)}(g)(Av)
\end{equation}
for all $g \in G_{n-r},\,u \in U_r,\,X \in M_{n-r \times r}$ and $v
\in V.$ For any subgroup $Y$ of $G_{n-r}$ let
\[
    H_{Y,r}=\{\left(
              \begin{array}{cc}
                y & X \\
                0 & u \\
              \end{array}
            \right):y \in Y,\,X \in M_{n-r \times r},\,u \in U_r\}.
\]
the map $A$ provides the identification
\begin{equation}\label{eq: derivative functionals}
    \Hom_{H_{Y,r}}(\pi,\psi_r)=\Hom_{Y}(\pi^{(r)},1).
\end{equation}

The functor of $m$th derivative satisfies a `Leibnitz rule'. In
\cite[Lemma 4.5]{MR58:28310}, it is proved that for representations
$\sigma_1$ of $G_{k}$ and $\sigma_2$ of $G_r$ the representation
$(\sigma_1 \times \sigma_2)^{(m)}$ is glued together from
$\sigma_1^{(i)} \times \sigma_2^{(m-i)},\,i=0,\dots, m,$ i.e.  that
there is a filtration on $(\sigma_1 \times \sigma_2)^{(m)}$ with
factors $\sigma^{(i)} \times \sigma_2^{(m-i)}$. An easy consequence
of \cite[Proposition 4.13 (a),(b)]{MR58:28310} is the following.
\begin{lemma}\label{lemma: der onto}
There exists a surjective morphism from $(\sigma_1 \times \sigma_2)^{(m)}$ to $\sigma_1 \times \sigma_2^{(m)}.$
\end{lemma}
\begin{proof}

It follows from \cite[Proposition 4.13 (a)]{MR58:28310} that there
is a surjective morphism
\[
    (\sigma_1
\times \sigma_2)_{|P_{k+r}} \twoheadrightarrow \sigma_1 \times
({\sigma_2}_{|P_r})
\]
(see \cite[p. 457-458]{MR58:28310} for the meaning of the induced
representation $\sigma_1 \times ({\sigma_2}_{|P_r})$). Let $\Omega$
be either $\Phi^-$ or $\Psi^-$. From \cite[Proposition 4.13
(b)]{MR58:28310} for a representation $\tau$ of $P_r$ we have
\[
    \Omega(\sigma_1 \times \tau)=\sigma_1 \times
    \Omega(\tau).
\]
Together with the exactness of $\Omega$ this implies that there is a
surjective morphism
\[
\Omega((\sigma_1 \times \sigma_2)_{|P_{k+r}}) \twoheadrightarrow
\sigma_1 \times \Omega({\sigma_2}_{|P_r}).
\]
Iterating this argument the lemma follows from the definition of the
derivative.
\end{proof}
If $\pi^{(k)} \not=0$ and $\pi^{(m)}=0$ for all $m>k$ then
$\pi^{(k)}$ is called the highest derivative of $\pi$. For any
segment $\Delta=[a,b]^{(\rho)}$ let $\Delta^-=[a,b-1]^{(\rho)}$. If
$a \in \oo$ is any multiset of segments , let $a^-$ be the multiset
of segments $\Delta^-$ so that $\Delta$ is a segment in $a$ and
$\Delta^-$ is not empty.
\begin{theorem}[\cite{MR584084}, Theorem 8.1]\label{highest derivative}
For all $a \in \oo$ the highest derivative of $\langle a \rangle$ is
$\langle a^- \rangle$.
\end{theorem}
\subsection{The unitary dual of $G$}
We briefly review the classification of Tadic for the unitary dual
of the general linear groups (\cite[Theorem D]{MR870688}). A
representation $\delta$ of $G_r$ is called square integrable if its
matrix coefficients belong to $L^2(Z_G \bs G)$. Denote by $D^u$ the
collection of all irreducible, square integrable representations of
$G_r$ with $r$ ranging over all positive integers. For $\delta \in
D^u$ and a positive integer $t$ the representation
\[
\delta[\frac{1-t}2]\times \delta[\frac{3-t}2]\times \cdots \times \delta[\frac{t-1}2]
\]
has a unique irreducible subrepresentation which we denote by
$U(\delta,t).$ It will be convenient to allow the notation
$U(\delta,0)$ for the trivial representation of $G_0=\{e\}$. Let $B$
be the collection of all representations of the form $U(\delta,t)$
or $U(\delta,t)[\alpha] \times U(\delta,t)[-\alpha]$ where $\delta
\in D^u,$ $t$ is a positive integer and $0<\alpha< \frac12$.
\begin{theorem}[Tadic]\label{tadic}
For every $\sigma_1,\,\dots,\sigma_t \in B$ the representation
$\sigma_1 \times \cdots \times \sigma_t$ is irreducible and unitary.
Any irreducible, unitary representation of $G$ is of this form
uniquely determined (up to reordering) by the multiset
$\{\sigma_1,\dots,\sigma_t\}$.
\end{theorem}
\subsection{Derivatives of Speh representations.} We will have to compute highest derivatives for
certain representations induced from Speh representations. In order
to be able to apply Theorem \ref{highest derivative}, we need to
express the representations of the form $U(\delta,t)[\alpha]$ in terms of the
classification of Zelevinsky for the dual of $G$, i.e. we wish to
describe explicitly the set $a=a(\delta,t,\alpha) \in \oo$ such that
$U(\delta,t)[\alpha]=\langle a\rangle.$
For any $\delta \in D^u$ there is an irreducible, cuspidal, unitary representation $\rho \in \ccc$ and a positive integer
$d$ such that $\delta$ is the unique irreducible
subrepresentation of
\[
    \rho[\frac{d-1}2] \times \cdots \times \rho[\frac{1-d}2],
\]
i.e. such that $\delta=\langle \{\rho[\frac{d-1}2],\,\rho[\frac{d-3}2],\dots, \rho[\frac{1-d}2]\}\rangle$
is given in terms of singleton segments. Let
\[
    \Delta(t,\rho)=[\frac{1-t}2,\frac{t-1}2]^{(\rho)}.
\]
In \cite[Theorem A. 10 (iii)]{MR870688} it is proved that
\[
    U(\delta,t)=\langle\{\Delta(t,\rho[\frac{1-d}2]),  \Delta(t,\rho[\frac{3-d}2]),\dots,
    \Delta(t,\rho[\frac{d-1}2]) \}\rangle.
\]
As pointed out in \cite[Theorem 3.2]{MR1359141}, for any $\alpha \in
\rr$ we then have
\[
    a(\delta,t,\alpha)=\{\Delta(t,\rho[\frac{1-d}2+\alpha]),  \Delta(t,\rho[\frac{3-d}2+\alpha]),\dots,
    \Delta(t,\rho[\frac{d-1}2+\alpha]) \}.
\]
Note that
\[
    \Delta(t,\rho)^-=\Delta(t-1,\rho[-\frac12])
\]
and therefore that
\[
    a(\delta,t,\alpha)^-=\{\Delta(t-1,\rho[\frac{1-d}2+(\alpha-\frac12)]),  \dots,
    \Delta(t-1,\rho[\frac{d-1}2+(\alpha-\frac12)]) \}.
\]
By Theorem \ref{highest derivative} we get that the highest derivative of $U(\delta,t)[\alpha]$ is
\begin{equation}\label{eq: speh high der}
    U(\delta,t-1)[\alpha-\frac12].
\end{equation}
See \cite[Theorem 3]{MR1062967} for an analog for $GL_n(\rr).$
Applying the Leibnitz rule for derivatives and an easy inductive
argument we obtain the following.
\begin{lemma}\label{lemma: speh prod der}
Let $\delta_,\dots,\delta_m \in D^u, \,t_1,\dots,t_m$ positive
integers and $\alpha_1,\dots, \alpha_m \in \rr$. The highest
derivative of the representation
\[
    U(\delta_1,t_1)[\alpha_1] \times \cdots \times U(\delta_m,t_m)[\alpha_m]
\]
is the representation
\[
    U(\delta_1,t_1-1)[\alpha_1-\frac12] \times \cdots \times U(\delta_m,t_m-1)[\alpha_m-\frac12].
\]
\end{lemma}
\begin{proof}
Let $\delta_i$
be a representation of $G_{r_i}$.
We wish to show that
\begin{multline*}
    (U(\delta_1,t_1)[\alpha_1] \times \cdots \times U(\delta_{m},t_{m})[\alpha_{m}])^{(k)}=\\
    \begin{cases} 0 & k>r_1+\cdots+r_{m} \\
    U(\delta_1,t_1-1)[\alpha_1-\frac12] \times \cdots \times U(\delta_m,t_m-1)[\alpha_m-\frac12] & k=r_1+\cdots+r_{m}.
    \end{cases}
\end{multline*}
We have already proved  this when $m=1$. Assume by induction that this is true for $m-1$.
By Leibnitz rule, the representation $(U(\delta_1,t_1)[\alpha_1] \times \cdots \times U(\delta_{m},t_{m})[\alpha_{m}])^{(k)}$ is glued together from
\[
    (U(\delta_1,t_1)[\alpha_1])^{(i)} \times (U(\delta_2,t_2)[\alpha_2]\cdots \times U(\delta_{m},t_{m})[\alpha_{m}])^{(k-i)}
\]
for $i=0,\dots, k$. But if $k>r_1+\cdots +r_m$ then for each $i$ we
have either $i>r_1$ or $k-i>r_2+\cdots+r_m$ and therefore it follows
from the induction hypothesis that  $(U(\delta_1,t_1)[\alpha_1]
\times \cdots \times U(\delta_{m},t_{m})[\alpha_{m}])^{(k)}=0$.
Similarly, if $k=r_1+\cdots+r_m$ then any component with $i \ne r_1$
must vanish. It follows that
\begin{multline*}
    (U(\delta_1,t_1)[\alpha_1] \times \cdots \times U(\delta_{m},t_{m})[\alpha_{m}])^{(r_1+\cdots+r_m)}=\\
    (U(\delta_1,t_1)[\alpha_1])^{(r_1)} \times (U(\delta_2,t_2)[\alpha_2]\cdots \times U(\delta_{m},t_{m})[\alpha_{m}])^{(r_2+\cdots+r_m)}
\end{multline*}
and the lemma follows by the induction hypothesis.
\end{proof}
\subsection{Klyachko Models for some representations of $G$}
We are now ready to state our main local result.
\begin{theorem}\label{thm: main local}
Let $\delta_1,\dots, \delta_q,\delta_1',\dots, \delta _{q'}' \in
D^u$, let $m_1,\dots,m_q,\,m_1',\dots, m_{q'}'$ be non negative
integers and let $\alpha_1,\dots, \alpha_q,\alpha_1',\dots,
\alpha_{q'}' \in \rr$. Assume that $\delta_i$ is a representation of
$G_{r_i}$, that $\delta_i'$ is a representation of $G_{r_i'}$ and
that
\[
n=\sum_{i=1}^q (2m_i+1)r_i+\sum_{i=1}^{q'}2m_i'r_i'.
\]
Let $n=r+2k$ where
\[
    r=r_1+\cdots+r_q \text{ and }k=m_1r_1+\cdots +m_qr_q+m_1'r_1'+\cdots
    +m_{q'}r_{q'}.
\]
The representation
\begin{multline}\label{eq: general form rep}
U(\delta_1',2m_1')[\alpha_1'] \times \cdots \times U(\delta_{q'}',2m_{q'}')[\alpha_{q'}']\\
 \times
U(\delta_1,2m_1+1)[\alpha_1] \times \cdots \times U(\delta_q,2m_q+1)[\alpha_q]
\end{multline}
is $(H_{r,2k},\psi_r)$-distinguished.
\end{theorem}
\begin{proof}
For $\delta \in D^u$ the contragradiant of $U(\delta,t)$ is
$U(\tilde\delta,t)$. Since the contragradiant of an induced
representation $\sigma_1 \times \cdots \times \sigma_t$ is $\tilde
\sigma_1 \times \cdots \times \tilde\sigma_t$ the contragradiant of
a representation that has the form \eqref{eq: general form rep} is
also of such a form. It therefore follows from Lemma \ref{lemma: H
to H'} that the theorem is equivalent to the statement that
representations $\pi$ of the form \eqref{eq: general form rep} are
$(H_{2k,r}',\psi_r')$-distinguished ($\psi$ and therefore $\bar\psi$
is an arbitrary non trivial character of $F$). Let
\[
\sigma_1 =U(\delta_1',2m_1')[\alpha_1'] \times \cdots \times
U(\delta_{q'}',2m_{q'}')[\alpha_{q'}'],
\]
\[
\sigma_2=U(\delta_1,2m_1+1)[\alpha_1] \times \cdots \times
U(\delta_q,2m_q+1)[\alpha_q]
\]
and $\pi=\sigma_1 \times \sigma_2.$ By Lemma \ref{lemma: der onto}
there is a surjective morphism $\mathfrak{p}:\pi^{(r)} \to \sigma_1
\times \sigma_2^{(r)}.$ There is also a surjective linear map
$A:\pi\to \pi^{(r)}$ satisfying the equivariance properties of
\eqref{derivative property}. By Lemma \ref{lemma: speh prod der} we
see that
\begin{multline*}
\sigma_1 \times \sigma_2^{(r)}=U(\delta_1',2m_1')[\alpha_1'] \times
\cdots \times U(\delta_{q'}',2m_{q'}')[\alpha_{q'}'] \times \\
U(\delta_1,2m_1)[\alpha_1-\frac12] \times \cdots \times
U(\delta_q,2m_q)[\alpha_q-\frac12]
\end{multline*}
is induced from Speh representations of the form
$U(\delta,t)[\alpha]$ with $t$ even. By \cite[Proposition 2]{mix},
there exists a non zero element $\ell \in Hom_{Sp(2k)}(\sigma_1
\times \sigma_2^{(r)},\cc).$ It follows that $\ell \circ
\mathfrak{p} \circ A$ is a non zero element of
$\Hom_{H_{2k,r}'}(\pi,\psi_r').$
\end{proof}
Using the notation of the statement of Theorem \ref{thm: main local} for a representation
$\pi$ of the form \eqref{eq: general form rep}
we define
\[
\kappa(\pi)=k.
\]
Note that by Theorem \ref{tadic} every irreducible, unitary representation $\pi$ of $G$ is of the form
\eqref{eq: general form rep} with $\abs{\alpha_i},\,\abs{\alpha_i'}<\frac12.$
In particular $\kappa(\pi)$ is defined.
The following corollary is then immediate from Theorem \ref{thm: main local}.
\begin{corollary}
Let $\pi$ be an irreducible, unitary representation of $G$ then
$\pi$ is
$(H_{n-2\kappa(\pi),2\kappa(\pi)},\psi_{n-2\kappa(\pi)})$-distinguished.
\end{corollary}
This proves in particular Theorem \ref{localmainintro}.

\end{document}